\documentclass[12pt]{amsart}
\usepackage{amsmath}
\usepackage{amsfonts}
\usepackage{amssymb}
\usepackage[utf8]{inputenc}
\usepackage[T1]{fontenc}

\newtheorem{thrm}{Theorem}[section]

   \newtheorem{fact}[thrm]{Proposition}
   \newtheorem{lemma}[thrm]{Lemma}
   \newtheorem{col}[thrm]{Corollary}
   \newtheorem{defn}[thrm]{Definition}
   
   \newtheorem{question}[thrm]{Question}
\newcommand{\AAA}{\mathcal{A}}
\newcommand{\KK}{\mathcal{K}}
\newcommand{\LL}{\mathcal{L}}
\newcommand{\FF}{\mathcal{F}}
\newcommand{\DD}{\mathcal{D}}
\newcommand{\CC}{\mathcal{C}}
\newcommand{\ZZ}{\mathcal{Z}}
\DeclareMathOperator{\mt}{mt}
\DeclareMathOperator{\ur}{sd}
\DeclareMathOperator{\conv}{conv}
\DeclareMathOperator{\MA}{MA}

\newcommand{\dx}{\;\text{d}}

\newcommand{\sub}{\subseteq}

\title[A dichotomy for the spaces of measures]{A dichotomy for the convex spaces\\ of probability measures}
\author[M.\ Krupski]{Miko\l aj Krupski}
\address{Instytut Matematyczny\\ Polska Akademia Nauk\\ Ul.\ \'{S}niadeckich 8\\00-956 Warszawa\\ Poland }
\email{krupski@impan.pl}

\author[G.\ Plebanek]{Grzegorz Plebanek}
\address{Instytut Matematyczny\\ Uniwersytet Wroc\l awski\\ Pl.\ Grunwaldzki 2/4\\
50-384 Wroc\-\l aw\\ Poland} \email{grzes@math.uni.wroc.pl}

\date{\today}


\thanks{G.\ Plebanek was partially supported by MNiSW Grant N N201 418939 (2010--2013).}

\begin{document}
\begin{abstract}
We show that every nonempty compact and convex space $M$ of probability Radon measures either contains
a measure which has  `small' local character in $M$ or else $M$ contains a measure of `large' Maharam type.
Such a dichotomy is related to several results on Radon measures on compact spaces and to some properties of Banach spaces of continuous functions.
\end{abstract}

\maketitle

\section{Introduction}

Throughout this note $K$ denotes a compact Hausdorff space. By a {\em Radon measure} $\mu$ on $K$ we
mean a finite measure defined on the Borel $\sigma$-algebra $Bor(K)$ of $K$ which is inner regular, that is
$$\mu(B)=\sup\{\mu(F):F=\overline{F}\subseteq B\},$$
for any $B\in Bor(K)$. We denote by $P(K)$ the space of all probability Radon measures on $K$. The space
$P(K)$ is always equipped with the $weak^*$ topology inherited from $P(K)\sub C(K)^*$, where
$C(K)$ is the  Banach space of continuous functions on $K$. When we treat a given measure $\mu\in P(K)$
as a functional on $C(K)$ we write $\mu(g)$ rather than $\int_K g\dx\mu$.

We say that a measure $\mu\in P(K)$ is of \textit{Maharam type} $\kappa$ and write  $\mt(\mu)=\kappa$
if $L_1(\mu)$ has density character $\kappa$.
Except for some trivial cases, $\mt(\mu)$ can be defined as the minimal cardinality of a family $\DD\sub Bor(K)$
such that for any $\varepsilon>0$ and any $B\in Bor(K)$ there is $D\in\DD$ such that
$\mu(D\bigtriangleup B)<\varepsilon$.

We consider here nonempty compact and convex subsets $M\sub P(K)$.
The dichotomy announced by the title of this note states, in particular, that
every such a space $M$ either has a $G_\delta$ point or contains a measure $\mu$ of uncountable
Maharam type. Roughly speaking, this means that to build a complicated
topological space like $M$ one needs large measures.

Subsequent sections illustrate how our dichotomy works. The condition that $\mu\in P(K)$ is
a $G_\delta$ point in $P(K)$ is equivalent to some measure-theoretic property of $\mu$, namely to
$\mu$ being strongly countably determined. It follows that every compact
space $K$ either carries a strongly countably determined measure (topologically simple, behaving
like measures on metric spaces) or a Radon measure of  uncountable type.

For a given compact space $K$, the existence of  Radon measures on $K$ of uncountable Maharam types is connected with the existence of continuous surjections from $K$ or $P(K)$ onto Tichonov cubes. This enables us to derive from our main theorem some purely topological results on compact spaces of measures. 

Finally, we give some applications of our dichotomy to Grothendieck-like properties
of Banach spaces of the form $C(K)$.

Our main theorem arose as a generalization of a result due to Borodulin-Nadzieja
\cite[Theorem 4.3]{PBN} and was further inspired by a theorem due to Haydon, Levy and Odell
  \cite[Corollary 3C]{HLO}). Although our dichotomy is a generalization of
 the latter, its proof seems to be more  straightforward.
\medskip

\textbf{Acknowledgment.} This paper was partially inspired by a series of lectures given by Piotr Koszmider at the Institute of Mathematics of the University of Wroc\l aw in November-December 2010; notes from those lectures can be found in \cite{Ko}.

\section{Dichotomy}

If $X$ is a topological space and $x\in X$ then $\chi(x,X)$ denotes the local character of $X$ at $x$,
i.e.\ the minimal cardinality of a local base at $x$. Recall that in a compact space
$X$ the local character at $x\in X$ agrees with the minimal cardinality of a family of open sets
intersecting to $\{x\}$.

In this section we shall present the main result of this paper stated as Theorem \ref{main} below.
We first recall the following useful fact due to Douglas (see \cite{D}, Theorem~1), which
is reproduced here together with its short proof.

\begin{lemma}\label{lemma}
Let $K$ be a  compact space and let $\FF\subseteq C(K)$ be any family of functions.
Suppose that measures $\mu_1,\mu_2\in P(K)$ satisfy $\mu_1\neq\mu_2$ and $\mu_1(f)=\mu_2(f)$
for every $f\in\FF$.

If $\mu=\frac{\mu_1+\mu_2}{2}$ then $\FF$ is not dense in $L_1(\mu)$.
\end{lemma}

\begin{proof}
We have $\mu=\frac{\mu_1+\mu_2}{2}$ so $\mu_1$ is absolutely continuous with respect to $\mu$, in fact
$\mu_1\le 2\cdot \mu$.  By the Radon-Nikodym theorem there exists $h\in L_\infty(\mu)$ such that
$\text{d}\mu_1=h\; \text{d}\mu$. Note that $1-h\neq 0$ since $\mu_1\neq\mu_2$.

For any function $f\in \FF$ we have
$$\int f(1-h)\dx\mu=\int f \dx\mu-\int f h \dx\mu=\int f \dx\mu-\int f \dx\mu_1=0,$$
which  proves that $\FF$ lies in the kernel of a nonzero continuous functional on $L_1(\mu)$ so is not dense in $L_1(\mu)$.
\end{proof}

\begin{thrm}\label{main}
Let $K$ be a compact space and  $\kappa$ be any cardinal number of uncountable cofinality.
If  $M$ is a  nonempty compact and convex subspace of $P(K)$ then either
\begin{itemize}
 \item[(D1)] there exists $\mu\in M$, such that $\chi(\mu,M)<\kappa$, or
 \item[(D2)] there exists $\mu\in M$, such that $\mt(\mu)\geqslant\kappa$.
\end{itemize}
\end{thrm}

\begin{proof}
Assuming  that for any $\mu\in M$ we have that $\chi(\mu,M)\geqslant\kappa$ we shall prove that
there is  a measure $\mu\in M$ with $\mt(\mu)\geqslant\kappa$.

We construct inductively a family $\{\mu_\alpha\in M:\alpha\le\kappa\}$
of measures from $M$ and a family $\{f_\alpha:\alpha< \kappa\}$ of functions from $C(K)$ such that, writing
for any $\alpha<\kappa$,
\[\FF_\alpha=
\{f_\xi:\xi\le\alpha\}\cup \{\lvert f_{\xi}-f_{\zeta}\rvert: \zeta<\xi\le \alpha\},\]
the following are satisfied

\begin{itemize}
\item[(i)] $\mu_\beta(g)=\mu_\alpha(g)$ whenever $\alpha<\beta\le\kappa$ and  $g\in \FF_\alpha$;\\[1ex]
\item[(ii)] for every $\beta<\kappa$
$$\varepsilon_\beta=\inf_{\alpha<\beta}\int\lvert f_\alpha-f_\beta\rvert \dx\mu_\beta>0.$$
\end{itemize}

Let us note first that once the construction is done,
the measure $\mu=\mu_\kappa$ will have the desired property. Indeed,
as $\kappa$ has uncountable cofinality, there is $\varepsilon>0$ such that $\varepsilon_\beta\ge\varepsilon$ for $\beta$ from some set $Y\subseteq\kappa$ of size $\kappa$.
Using (i),  for any $\beta,\beta'\in Y$, $\beta'<\beta<\kappa$ we have
\[\int\lvert f_{\beta}-f_{\beta'}\rvert \dx\mu_\kappa=\int\lvert f_{\beta}-f_{\beta'}\rvert \dx\mu_{\beta}>\varepsilon_\beta\ge \varepsilon,\]
so the family $\{f_\beta: \beta\in Y\}$ witnesses that $\mt(\mu_\kappa)\geqslant\kappa$.

We shall now describe the inductive step:
let us assume that we have already constructed measures $\{\mu_\alpha\in M:\alpha<\beta\}$ and continuous functions $\{f_\alpha:\alpha<\beta\}$ so that conditions (i)-(ii) are satisfied.

Let us denote $\FF^*=\bigcup_{\alpha<\beta}\FF_\alpha$ and for every $\alpha<\beta$ write
\[M_{\alpha}=\{\nu\in M:\; \nu(f)=  \mu_\alpha(f) \mbox{ for all } f\in \FF_\alpha\}.\]
It follows from  (i) that $\{M_\alpha: \alpha<\beta\}$  is a decreasing family of nonempty and closed sets in $M$ so, by compactness, 
$M_\beta=\bigcap_{\alpha<\beta}M_\alpha$ is nonempty, too.

For fixed $f\in \FF^*$ and  $\alpha<\beta$, the set
$\{\nu\in M: \nu(f)=\mu_\alpha(f)\}$ is a $G_\delta$ subset of $M$.
It follows that $M_{\beta}$ is an intersection of less than  $\kappa$ open sets in $M$
(since $|\FF^*|<\kappa$), so $M_{\beta}$ cannot consist of  a single point (by our assumption that (D1) does not hold).

Take $\nu',\nu''\in M_{\beta}$, $\nu'\neq\nu''$
and put $\mu_{\beta}=\frac{\nu'+\nu''}{2}$.
By Lemma \ref{lemma} $\FF^*$ is not dense in $L_1(\mu_{\beta})$ so there is a continuous function $f_{\beta}\in C(K)$ such that
\[\inf_{\xi<\beta}\int\lvert f_\xi-f_{\beta}\rvert \dx\mu_{\beta}>0,\]
and we are done.
\end{proof}

Let us note that the alternatives of Theorem \ref{main}  have rather different content:
(D1) is a topological statement on the space $M$ while (D2) names a  measure-theoretic property of elements of $M$. We shall see below, however, that in fact (D1) is equivalent to another purely measure-theoretic property. Likewise, we
shall mention instances when (D2) has natural topological consequences.

\section{Strongly determined measures}

A measure $\mu\in P(K)$ is said to be \textit{strongly countably determined} if there exists a continuous map
$f:K\rightarrow [0,1]^\omega$, such that for any compact set $F\subseteq K$ we have
$\mu(F)=\mu(f^{-1}f[F])$.
Strongly countably determined  measures were introduced by Babiker \cite{B}
who gave them the name {\em uniformly regular measures},
motivated by considerations related to uniform spaces. A measure $\mu$ is strongly countably determined
if and only if there is a countable family $\ZZ$ of closed $G_\delta$ subsets of $K$ such that
\[(*)\quad \mu(U)=\sup\{\mu(Z): Z\in\ZZ, Z\subseteq U\},\]
for every open set $U\subseteq K$, see \cite{B}. The latter condition seems to justify the name
strongly countably determined; note that
if we relax the property to saying that $\ZZ$ is a countable family of closed sets then a measure
$\mu$ satisfying (*)  is called {\em countably determined}.
(Strongly) countably determined measures we considered by
Pol \cite{P}, Mercourakis \cite{Me96}, Plebanek \cite{Pl00} and Marciszewski \& Plebanek \cite{MP11}.

We shall consider the following more general notion  (see \cite{K}).

\begin{defn}\label{def}
Let us say that a measure $\mu\in P(K)$ is strongly $\kappa$-determined if
there exists continuous map
$f:K\rightarrow [0,1]^\kappa$ such that for any compact set $F\subseteq K$
we have $\mu(F)=\mu(f^{-1}f[F])$.
 We shall write $\ur(\mu)=\kappa$ to denote the least cardinal $\kappa$ for which
$\mu$ is strongly $\kappa$-determined.
\end{defn}

It can be checked  that, as in the case $\kappa=\omega$,
$\ur(\mu)$ is the minimal cardinality of an infinite family $\ZZ$ of closed $G_\delta$ sets
such that for any open set $U\subseteq K$ and any $\varepsilon>0$, there is $Z\in\ZZ$, $Z\subseteq U$ and
$\mu(U\setminus Z)<\varepsilon$.

Clearly $\mt(\mu)\le\ur(\mu)$, since the family approximating all open sets as in (*)
needs to $\bigtriangleup$-approximate all Borel sets by regularity of a measure.
Note that the Dirac measure $\mu=\delta_x$, where $x\in K$, is strongly countably determined if and only if
$x$ is a $G_\delta$ point in  $K$; in fact we always have $\ur(\delta_x)=\chi(x,K)$.

In \cite{P} Pol proved that if a measure $\mu\in P(K)$ is strongly countably determined then $\chi(\mu,P(K))=\omega$. A straightforward modification of
the argument used in \cite[Proposition 2]{P} yields the following.

\begin{fact}\label{fact:Pol}
For any compact space $K$ and a measure $\mu\in P(K)$ we have $\chi(\mu,P(K))\le \ur(\mu)$.
\end{fact}

We shall check  now that in fact the equality $\chi(\mu,P(K))=\ur(\mu)$ always holds.

For the sake of the next lemma we recall some standard fact concerning
finitely additive measures on algebras of sets. Suppose that $\AAA$ is an algebra of subsets
of some space $X$. If $\LL\sub\AAA$ is a lattice (that is $\LL$ is closed under finite unions and
intersecions) then a finitely additive measure $\mu$ on $\AAA$ is said to be $\LL$-regular
if 
\[\mu(A)=\sup\{\mu(L): L\sub A, L\in\LL\},\]
for every $A\in\AAA$. The following result can be found in Bachman \& Sultan \cite{BS}, Theorem 2.1.

\begin{thrm}\label{lemma:extension}
Let $\mathcal{A}\subseteq P(X)$ be an algebra of sets. Let $\LL\subseteq \mathcal{A}$ be a lattice
and let $\mu$ be an $\LL$-regular, finitely additive measure on $\mathcal{A}$. Assume $\KK\supseteq\LL$
is also a lattice. Then $\mu$ extends to a $\KK$-regular, finitely additive measure $\nu$ on $alg(\mathcal{A}\cup\KK)$.
\end{thrm}

If $f:K\rightarrow L$ is a surjective map between compact spaces and $\mu\in P(K)$ is a measure on $K$, then by $f[\mu]$ we denote
the image measure  $f[\mu]\in P(K)$ which is for $A\in Bor(K)$ defined as  $f[\mu](A)=\mu(f^{-1}[A])$.

\begin{lemma}\label{lemma1}
Let $g:K\rightarrow L$ be a continuous surjection between  compact spaces $K$ and $L$
and let $F\subseteq K$ be a closed set.

Given  $\lambda\in P(L)$, there exist a measure $\mu\in P(K)$ such that $g[\mu]=\lambda$
and $\mu(F)=\lambda(g[F])$.
\end{lemma}

\begin{proof}
Let
$$\AAA=\{g^{-1}[B]: B\in Bor(L)\},\quad \LL=\{g^{-1}(H):H=\overline{H}\subseteq L \}.$$
Then $\AAA$ is an algebra of sets, $\LL\sub \AAA$ is a lattice and putting
$\mu_0(g^{-1}[B])=\lambda(B)$ we have an $\LL$-regular measure $\mu_0$ defined on $\AAA$.

Let $\LL'$ be the lattice generated by $\LL\cup\{F\}$ and $\AAA'=alg(\AAA\cup\{F\})$.
Using Theorem \ref{lemma:extension} we can extend $\mu_0$ to a finitely additive measure
$\mu'$ on $\AAA'$ which is $\LL'$-regular.
Note that we then have
\[\mu'(F)=\mu_0^*(F)=\inf\{\lambda(B): F\sub g^{-1}[B]\}=\lambda(g[F]).\]
Finally, $\mu'$ can be extended to a finitely additive closed-regular measure $\mu$ on
an algebra of sets containing $\AAA'$ and all closed subsets of $K$.
Such $\mu$ is then countably additive and extends to a Radon measure by a standard measure
extension theorem, see \cite{BS} for details.
\end{proof}

\begin{thrm}
If  $K$ is a compact space  and $\mu\in P(K)$ then $\chi(\mu,P(K))=\ur(\mu)$.
\end{thrm}

\begin{proof}
By Proposition \ref{fact:Pol}, we only need to show $\chi(\mu,P(K))\ge\ur(\mu)$.

Fix $\mu\in P(K)$ and  denote $\chi(\mu,P(K))=\kappa$. We can find a family
$\{f_\xi: \xi<\kappa\}$ of continuous functions $f_\xi:K\to[0,1]$ such that
whenever $\nu\in P(K)$ and $\nu(f_\xi)=\mu(f_\xi)$ for all $\xi<\kappa$ then $\nu=\mu$.

Consider now the diagonal mapping
\[f=\Delta_{\xi<\kappa}:K\rightarrow [0,1]^\kappa,\quad f(x)(\xi)=f_\xi(x).\]
If we  suppose that $\ur(\mu)>\kappa$ then there is a closed set $F\subseteq K$ such that $\mu(f^{-1}f[F])>\mu(F)$. Let $\lambda=f[\mu]\in P([0,1]^\kappa)$;
by Lemma \ref{lemma1} there is  a measure $\nu\in P(K)$, such that $f[\nu]=\lambda$ and $\nu(F)=\lambda(f[F])=\mu(f^{-1}f[F])>\mu(F)$, which in particular means that $\nu\neq\mu$.

On the other hand, $\lambda=f[\mu]=f[\nu]$ so for every $\xi<\kappa$
\[f_\xi[\mu]=\pi_\xi\circ f[\mu]=\pi_\xi\circ f[\nu]=f_\xi[\nu],\]
where $\pi_\xi:[0,1]^\kappa\to [0,1]$ is a projection, and therefore
\[\int f_\xi\dx \mu=\int_0^1 t \dx f_\xi[\mu](t)=\int_0^1 t \dx f_\xi[\nu](t)= \int f_\xi\dx \nu,\]
which is a contradiction.
\end{proof}

Recall that a topological space $K$  is scattered of every subset of $K$ has an isolated point.
Nonscatteredness of $K$ is equivalent to the existence of a continuous surjection $K\to [0,1]$ and
to the existence of a nonatomic measure on $K$.

\begin{col}
Let $\kappa$ be a cardinal number of uncountable cofinality. If $K$ is any compact space
which is not  scattered then
either there a nonatomic measure $\mu\in P(K)$ such that $\ur(\mu)<\kappa$ or
there is a (nonatomic) measure $\mu\in P(K)$ with $\mt(\mu)\ge \kappa$.
\end{col}

\begin{proof}
If $K$ is not scattered then there is a continuous surjection $g:K\to [0,1]$. Writing $\lambda$
for the Lebesgue measure on $[0,1]$ we consider the set 
$$M=\{\mu\in P(K): g[\mu]=\lambda\}.$$
 Then
$M$ is a compact and convex set which is $G_\delta$ in $P(K)$. Applying Theorem \ref{main} we get
the result.
\end{proof}

For $\kappa=\omega$ we get the following generalization of
a result due to Borodulin-Nadzieja \cite[Theorem 4.6]{PBN}, which was stated in the context
of Boolean algebras or zero-dimensional spaces.

\begin{col}\label{Nadzieja}
Every compact nonscattered space  either carries a nonatomic stron\-gly countably determined measure
or a Radon measure of uncountable type.
\end{col}

Of course a compact space can carry both strongly countably determined measures and measures of uncountable
type. Note that infinite products such as $K=2^\kappa$, $\kappa>\omega$,
are examples of spaces not admitting
strongly countably determined measures. Other examples of $K$ with such a property are mentioned below in connection with the Grothendieck property.

There are several classes of spaces carrying only measures of countable type, which necessarily admit
strongly countably determined measures; this list of such classes includes

\begin{enumerate}
\item Eberlein compacta (i.e.\ spaces homeomorphic to weakly compact subsets of Banach spaces);
\item ordered compact spaces;
\item Corson compacta under $\MA(\omega_1)$;
\item consistently, all first-countable compact spaces, see Plebanek \cite{Pl00}
\item Rosenthal compacta, see Todor\v{c}evi\'c \cite{To99}, cf.\ Marciszewski \& Plebanek \cite{MP11}.
\end{enumerate}

The reader may consult Mercourakis \cite{Me96} for basic facts and further references concerning
(1)--(3).
In connection with (4) note that it is relatively consistent that
every Radon measure on a first-countable compactum
is strongly countably determined (\cite{Pl00}) but there is an open problem posed by D.H.\ Fremlin if this
is a consequence of $\MA(\omega_1)$.

\section{Mappings onto cubes}

In \cite{F} Fremlin proved assuming $\MA(\omega_1)$ that if a compact space $K$
carries a Radon measure of type $\ge\omega_1$ , then there is a continuous surjection
$f:K\rightarrow [0,1]^{\omega_1}$; cf.\ Plebanek \cite{Pl00} for further discussion.
As we have mentioned uncountable products are typical examples of spaces on which there
are no strongly countably determined measures;
combining Fremlin's result  with Theorem \ref{main} we get the following partial converse.

\begin{col}
Assuming $\MA(\omega_1)$, if a compact space $K$ admits no strongly countably determined measure
then there is a continuous surjection $f:K\rightarrow [0,1]^{\omega_1}$.
\end{col}

If there is a continuous surjection from a compact space $K$ onto $[0,1]^\kappa$ for some $\kappa$ then it
can be extended to a continuous mapping $P(K)\to [0,1]^\kappa$. The existence of the latter surjection is
also closely related to types of measures on $K$; the following is due to Talagrand \cite{Ta81}
and requires no additional set-theoretic assumptions.

\begin{thrm}
If $\kappa\ge\omega_2$ is a cardinal of uncountable cofinality and $K$ is a compact space then there
is a continuous surjection $P(K)\to [0,1]^\kappa$ if (and only if) $K$ carries  a Radon measure
of type $\ge\kappa$.
\end{thrm}

The above result and Theorem \ref{main} yield immediately the following.

\begin{col}\label{4:3}
If $\kappa\ge\omega_2$ is a cardinal of uncountable cofinality then for a compact space $K$ either
\begin{itemize}
\item[(i)] $P(K)$ has points of character $<\kappa$, or
\item[(ii)] $P(K)$ can be continuously mapped onto $[0,1]^\kappa$.
\end{itemize}
\end{col}

The reader is referred to a survey paper \cite{Pl02} for further discussion. We shall only mention here
 a problem related to Corollary \ref{4:3}. Let us recall that the tightness of a topological space $X$ 
is the minimal cardinal number $\tau(X)$ such that whenever 
$A\subseteq X$ and $x\in\overline{A}$ then
there is $I\subseteq A$ such that $|I|\le\tau(X)$ and  $x\in\overline{I}$.
 
If $K$ is such a compact space that $P(K)$ has  tightness $\le\omega_1$ then $P(K)$ cannot be continuously
mapped onto $[0,1]^{\omega_2}$ because the tightness of the latter space is $\omega_2$ and tighness
is not increased by continuous surjections between compact spaces; hence
Corollary \ref{4:3} implies in particular that  then every measure $\mu\in P(K)$ has the Maharam type at most
$\omega_1$. We do not know if the following holds true.

\begin{question}
Assume that $K$ is such a compact space that $P(K)$ has  
countable tightness.  Does this imply that every measure $\mu\in P(K)$ has countable Maharam type?
\end{question}

It seems that the only result in that direction, requiring no additional set-theoretic assumptions,
is a theorem stating that if $K$ is Rosenthal compact then $K$ admits only measures of countable type;
the result was obtained by Jean Bourgain but the first published proof is due to 
Todor\v{c}evi\'c \cite{To99}; see also \cite{MP11}.

\section{On  Haydon-Levy-Odell result}

If $(v_n)_n$ is a sequence in some vector space then a sequence $(w_n)_n$ is said to be a
{\em convex block subsequence} of $(v_n)_n$ if for some sequence of natural numbers $k_1<k_2<\ldots$,
every $w_n$ is a convex combination of vectors $v_i$ for $k_n\le i < k_{n+1}$. 
The following result was proved by  Haydon, Levy and Odell (see \cite{HLO}, Corollary 3C).

\begin{thrm}\label{HLO}
If $K$ is compact and in $P(K)$ there is a sequence $(\lambda_n)$, with no weak* convergent convex block subsequence, then there is $\mu\in P(K)$ such that $\mt(\mu)\ge\mathfrak{p}$.
\end{thrm}

Recall that the cardinal number $\mathfrak{p}$ mentioned here  is the largest cardinal having the property:
\medskip

\textit{If $\kappa<\mathfrak{p}$ and $(M_\alpha)_{\alpha<\kappa}$ is a family of subsets of $\omega$ with $\bigcap_{\alpha\in F} M_\alpha$
infinite for all finite $F\subseteq\kappa$, then there exists an infinite $M\subseteq\omega$ with $M\setminus M_\alpha$ finite for all
$\alpha<\kappa$.}
\medskip

Theorem \ref{HLO} has several interesting consequences, for instance it is one of the main ingredients
of the proof that assuming $\mathfrak{p}=\mathfrak{c}>\omega_1$, for every infinite compact space
$K$, the Banach space $C(K)$ has either $l_\infty$ or $c_0$ as a quotient, see \cite{HLO}, cf.\ Koszmider
\cite{Ko} for  more information.

We note here that one can easily derive Theorem \ref{HLO} from our dichotomy as follows.

\begin{proof} (of Theorem \ref{HLO}).
Set
$$M=\bigcap_{n=1}^\infty\overline{\conv}\{\lambda_k:k\ge n\}.$$
By Theorem \ref{main} applied to such a set $M$ and $\kappa=\mathfrak{p}$ it is sufficient to check
that $\chi(\mu,M)\ge\mathfrak{p}$ for every $\mu\in M$.
But if we suppose that $\chi(\mu,M)<\mathfrak{p}$ for some $\mu\in M$ then $\mu$ is a limit
of some convex block subsequence of $(\lambda_n)_n$, which can be derived from the definition
of $\mathfrak{p}$, as indicated in Lemma \ref{lemma:p} below.
\end{proof}

\begin{lemma}\label{lemma:p}
Let $(M_n)_n$ be a decreasing sequence of separable topological spaces and let
$M=\bigcap_{n=1}^\infty M_n$. If $x$ is a nonisolated point in $M$ and
$\chi(x,M)<\mathfrak{p}$, then there is a sequence $(x_n)_n$ in $M_1$, convergent to $x$ and such that for any $k\in\omega$, $x_n\in M_k$ for all but finitely many $n$.
\end{lemma}

\begin{proof}
Since each $M_n$ is separable, there is a countable set $D\subseteq M_1$, such that for any $n\in \omega$, $D\cap M_n$ is dense in $M_n$.
Let  $\{U_\alpha:\alpha<\chi(x,M)\}$ a family of open subsets of $M_1$
 such that $\{U_\alpha\cap M:\alpha<\chi(x,M)\}$ is a local base at $x\in M$.
Consider $$\CC=\{U_\alpha\cap D:\alpha<\chi(x,M)\}\cup\{M_n\cap D:n<\omega\}.$$
Since $\chi(x,M)<\mathfrak{p}$, there exists an infinite $A\subseteq D$ with $A\setminus C$ finite for any $C\in \CC$. Now elements of $A$ form a desired sequence converging to $x$.
\end{proof}

Recall that a Banach space $X$ is called a \textit{Grothendieck space} if every weak$^*$ convergent sequence in the dual space $X^\ast$ is also weakly convergent. 
Note that if $K$ is an infinite compact space and
the Banach space $C(K)$ is Grothendieck then for every  sequence $(\mu_n)_n$ from $C(K)^*$ which is 
$weak^*$ convergent to $\mu\in C(K)^*$   one has $\mu(B)=\lim_n\mu_n(B)$ for all $B\in Bor(K)$.

It follows from Theorem \ref{HLO} that
$\mt(\mu)\ge\mathfrak{p}$ for some $\mu\in P(K)$: indeed, if $(x_n)_n$ is a sequence of distinct
elements of $K$ then it is not difficult to check that a sequence of measures
$(\delta_{x_n})_n$ has no convex block subsequence which converges weakly. Hence, by the Grothendieck
property and Theorem \ref{HLO}, $K$ must carry a Radon measure of type $\ge\mathfrak{p}$; this fact
is a particular case of Haydon's result from \cite{H}.

\begin{fact}
If $K$ is a compact space without isolated points and $C(K)$ is a Grothendieck space
then
\begin{itemize}
\item[(a)] $\ur(\mu)\ge\omega_1$ for every $\mu\in P(K)$; and
\item[(b)] $\ur(\mu)\ge\mathfrak{p}$ for every $\mu$ which is supported by a separable subspace of $K$.
\end{itemize}
\end{fact}

\begin{proof}
We shall discuss (b); one gets (a) by a simple modification of an argument below.
Suppose that $\kappa=\ur(\mu) < \mathfrak{p}$ for some $\mu\in P(K)$.

Let $\mu$ have an atom, say $c=\mu(\{x_0\})>0$ for $x_0\in K$. If $U$ is an open set containing $x_0$
then  there is an open set $V$ with $x_0\in V\sub U$ and $\mu(V\setminus\{x_0\})<c/2$. It follows that
if $\ZZ$ is a family of closed $G_\delta$ sets witnessing that $\ur(\mu)=\kappa$ then
there is $Z\in\ZZ$ such that $Z\sub V$ and $\nu(V\setminus Z)<c/2$ which implies $x_0\in Z\sub V$.
Using this remark we conclude that $\chi(x_0, K)\le \kappa$;
as $\kappa<\mathfrak{p}$ and $x_0$ is not isolated, there is a sequence $x_n \in K\setminus\{x_0\}$
converging to $x_0$, a contradiction with the Grothendieck property.

Let $\mu$ be nonatomic and $\mu(K_0)=1$ for some separable $K_0\sub K$; take
a countable dense subset $A$ of $K_0$.
Then $\mu$ is in the closure of a countable set
\[\{\sum_{i=1}^n r_i \delta_{a_i}: r_i\in\mathbb{Q}, a_i\in A\},\]
which in view of $\kappa=\ur(\mu)=\chi(\mu, P(K)$ implies that $\mu$ is a $weak^*$ limit
of a sequence of purely atomic measures; a contradiction since such a sequence cannot converge weakly.
\end{proof}


\begin{thebibliography}{99}
\bibitem{B}
A.G. Babiker, \textit{On uniformly regular topological measure spaces},
Duke Math.\ J.\ 43 (1976), 775-789.
\bibitem{BS}
G. Bachman, A. Sultan, \textit{On regular extensions of measures}, Pacific J.\ Math.\ 86 (1980), 389-394.
\bibitem{PBN}
P. Borodulin-Nadzieja, \textit{Measures on minimally generated Boolean algebras},
Topology Appl.\ 154 (2007), 3107-3124.
\bibitem{D}
R. G. Douglas, \textit{On extremal measures and subspace density},
Michigan Math.\ J.\ 11 (1964), 243-246.
\bibitem{F}
D.H.\ Fremlin, \textit{On compact spaces carrying Radon measures of uncountable Maharam type},
Fund.\ Math.\ 154 (1997), 295-304.
\bibitem{H}
R.\ Haydon, \textit{An unconditional result about Grothendieck spaces},
Proc.\ Amer.\ Math.\ Soc.\ 100 (1987), 511-516.
\bibitem{HLO}
R.\ Haydon, M.\ Levy, E.\ Odell, \textit{On sequences without weak* convergent convex block subsequences}, Proc.\ Amer.\ Math.\ Soc.\ 100 (1987), 94-98.
\bibitem{Ko}
P.\ Koszmider, \textit{Set-theoretic methods in Banach spaces}, preprint(2010) \newline \textit{http://www.im0.p.lodz.pl/~pkoszmider/dydaktyka/2010settheory.html}
\bibitem{K}
M.\ Krupski, \textit{Miary na przestrzeniach zwartych spe\l niaj\c{a}cych pierwszy aksjomat
przeliczalno\'sci}, Master Thesis, University of Wroc\l aw (2010).
\bibitem{MP11} W.\ Marciszewski, G.\ Plebanek, {\em On measures on Rosenthal compacta},
preprint (2010).
\bibitem{Me96} S.\ Mercourakis, {\em Some remarks on countably determined
measures and uniform distribution of sequences},\ Monats.\ Math.\ 121 (1996),
79--101.
\bibitem{Pl00} G.\ Plebanek, {\em Approximating Radon measures on first--countable compact spaces},
Colloq.\ Math.\  86 (2000), 15--23.
\bibitem{Pl02}
G.\ Plebanek, {\em
On compact spaces carrying Radon measures of large Maharam type}, Acta Univ. Carolinae 43 (2002), 87--99.
\bibitem{P}
R.\ Pol, \textit{Note on the spaces $P(S)$ of regular probability measures whose topology is determined by countable subsets}, Pacific J.\ Math.\ 100 (1982), 185-201.
\bibitem{Ta81} M.\ Talagrand, {\em Sur les espace de Banach contenant
$l^{\tau}$}, Israel J.\ Math.\ 40 (1981), 324--330.
\bibitem{To99} S.\ Todor\v{c}evi\'c, {\em Compact sets of the first Baire class},
J.\ Amer.\ Math.\ Soc.\ 12 (1999), 1179--1212.
\bibitem{V}
J. Vaughan, \textit{Small Uncountable Cardinals and Topology}, Open Problems in Topology (J. van Mill and G. Reed, eds) North-Holland, Amsterdam (1990).
\end{thebibliography}
\end{document}